\documentclass[12pt]{article}
\usepackage{amsthm}
\usepackage{plain}
\begin{document}
\theoremstyle{plain}
	\newtheorem{theorem}{Theorem}
	\newtheorem{lemma}{Lemma}
	\newtheorem{proposition}{Proposition}
\def\N{{\bf N}}
\def\Z{{\bf Z}}
\def\R{{\bf R}}

\title{Continuous Data Assimilation Using\\
	General Interpolant Observables}
\date{April 3, 2013}
\author{
Abderrahim Azouani\thanks{
		Freie Universit\"at Berlin,
		Institute f\"ur Mathematik I,
		Arnimallee 7, Berlin, Germany.
		{\it email:}{\tt\ azouani@math.fu-berlin.de}} \and
Eric Olson\thanks{
		Department of Mathematics and Statistics,
		University of Nevada,
		Reno, NV 89557, USA. {\it email:}{\tt\ ejolson@unr.edu}} \and
Edriss S. Titi\thanks{
		Department of Mathematics and
			Department of Mechanical and Aerospace Engineering,
		University of California,
		Irvine, CA 92697--3875, USA.
		{\it email:}{\tt\ etiti@math.uci.edu}}
	\thanks{
		The Department of Computer Science and Applied Mathematics,
		Weizmann Institute of Science,
		Rehovot 76100, Israel.
		{\it email:}{\tt\ edriss.titi@weizmann.ac.il}}
}

\maketitle

\begin{abstract}
We present a new continuous data assimilation algorithm based on ideas
that have been developed for designing finite-dimensional feedback
controls for dissipative dynamical systems, in particular,
in the context of the incompressible two-dimensional
Navier--Stokes equations.
These ideas are motivated by the fact that dissipative dynamical
systems possess finite numbers of determining parameters (degrees of freedom)
such as modes, nodes and local
spatial averages which
govern their long-term behavior.
Therefore, our algorithm allows the use of any type of measurement
data for which a general type of approximation interpolation operator exists.
Our main result provides conditions, on
the finite-dimensional spatial resolution of the collected data,
sufficient to guarantee that the approximating
solution, obtained by our algorithm
from the measurement data,
converges to the unknown reference solution over time. Our algorithm is also applicable in the context of signal synchronization in which  one can recover, asymptotically in time, the solution (signal) of the underlying dissipative system that is corresponding to a  continuously transmitted partial data.


\end{abstract}

\section{Introduction}

The goal of continuous data assimilation, and signal synchronization, is to use low spatial resolution
observational measurements, obtained continuously in time, to accurately
find the corresponding reference solution from which future predictions
can be made.
The motivating application of continuous data assimilation is
weather prediction.  The classical method of continuous data
assimilation, see Daley \cite{daley1991}, is to insert observational
measurements directly into a model as the latter is being integrated in time.
We propose a new approach based on ideas from
control theory, see Azouani and Titi \cite{azouani2013}.
A slightly similar approach in the context of stochastic
differential equations, using the low Fourier modes as observables/measurements, appears in a recent work
by Bl\"omker, Law, Stuart and Zygalakis \cite{stuart2012}.
Rather than inserting
the measurements directly into the model, i.e. into the nonlinear term, we introduce a feedback control
term that forces the model toward the reference
solution that is corresponding to the observations.
This is motivated by the fact that the measured data is usually obtained
as the values of the exact solutions at a discrete set of spatial
nodal points, and that it difficult to insert this data
directly into the underlying equation, because
it is not possible to obtain the exact values of
the spatial derivatives.
One should observe that in order to guarantee a unique
corresponding reference solution one has to supply
observational data with enough
spatial resolution.
This is the object of this paper.

While the classical method of continuous data assimilation is simple
in concept, special care has to be taken
concerning how the observations are inserted into a model in practice.
For example,
it is generally necessary to separate the fast and slow parts of a
solution before inserting the observations into the model.
The method proposed here does not require such a decomposition.
Since the observations
are not directly inserted into the model, we can rely on the
dissipation already present in the dynamics to filter the observed data, i.e. the viscous term will suppress the ``spill over" oscillations in the fine scales.
The advantage of this approach is that it works for a general
class of interpolant observables without modification.

Let $u(t)$ represent the state at time $t$ of the dynamical
system in which we are interested, and let $I_h(u(t))$ represent
our observations of this system at a coarse spatial resolution
of size $h$.
%The parameter $h$ is a
%length scale that corresponds to the spatial resolution of the observations.
Given observational measurements, $I_h(u(t))$, for $t\in [0,T]$,
our goal is to construct an increasingly accurate initial condition
from which predictions of $u(t)$, for $t>T$, can be made.
We do this by constructing an approximate solution $v(t)$ that
converges to $u(t)$ over time.

Suppose the time evolution of $u$ is governed by a
given evolution equation of the form
\begin{equation}
	{du\over dt} = F(u),
	\label{E1}
\end{equation}
where the initial data, $u_0$, is missing.
Our algorithm for constructing $v(t)$ from the observational
measurements $I_h(u(t))$ for $t\in[0,T]$ is given by
\begin{equation}
	{dv\over dt} = F(v) -\mu I_h(v) +\mu I_h(u),
    \label{approxeqn}
\end{equation}
\begin{equation}
	\qquad v(0)=v_0,
\end{equation}
where $\mu$ is a positive relaxation parameter,
which relaxes the coarse spatial scales of $v$
toward those of the observed data,
%with dimensions of inverse time
and $v_0$ is taken to be arbitrary.
It is worth stressing that our algorithm is designed to work for
general dissipative dynamical systems of the form (\ref{E1}).
Such systems are known to have finitely many degrees
of freedom in the form of determining parameters
of the type $I_h(u)$, see, for example,
Cockburn, Jones and Titi \cite{cockburn1996},
Foias, Manley, Rosa and Temam \cite{foias2001},
Foias and Prodi \cite{FP},
Foias and Temam \cite{FT1}, \cite{FT2},
Jones and Titi \cite{jones1992}, \cite{jones1993}, and references therein.
The incompressible two-dimensional Navier--Stokes equations provide
a concrete example of a dissipative dynamical system of this type.

We consider here the incompressible two-dimensional Navier--Stokes
equations, as a paradigm,  because they are amenable to mathematical
analysis while at the same time similar to the equations
used in realistic weather models.
Thus, we shall suppose the evolution of $u$ is
is governed by the Navier--Stokes system
\begin{equation} {\partial u\over \partial t} -\nu \Delta u
	+(u\cdot \nabla)u+\nabla p = f
	\label{nseq}
\end{equation}
\begin{equation}
	\nabla\cdot u = 0,
	\label{incomp}
\end{equation}
in the physical domain $\Omega$, with either no-slip Dirichlet, or
periodic, boundary conditions.
Here $u(x,t)$ represents velocity of the fluid at time $t$ at
position $x$, $\nu>0$ represents the kinematic viscosity,
$p(x,t)$ is the pressure and $f(x,t)$ is a time dependent body
force applied to the fluid.

%No-slip and periodic boundary conditions may be described as follows.
In the case of no-slip Dirichlet boundary conditions we take $u=0$
on $\partial \Omega$. The domain $\Omega$ is an open,
bounded and connected set in $\R^2$ with $C^2$ boundary,
such that $\partial \Omega$ can be represented locally as the
graph of a $C^2$ function.
In the case of periodic boundary conditions we require $u$ and $f$
to be $L-$periodic, in both $x$ and $y$ directions, and take
$\Omega=[0,L]^2$ to be the fundamental periodic domain.

Continuous data assimilation, in the context of the incompressible
two-dimensional
Navier--Stokes equations, was first studied
by Browning, Henshaw and Kreiss in \cite{browning1998},
later by Henshaw, Kreiss and Ystr\"om in \cite{henshaw2003}
and also by Olson and Titi in \cite{olson2003} and \cite{olson2008},
motivated by the concept of finite number of determining
modes which was introduced for the first time in \cite{FP} (see also \cite{foias2001}, \cite{olson2003}, and references therein).
These studies treated the case of periodic boundary conditions,
where the observations were given by the low
Fourier modes with wave numbers $k$, such that $|k|\le 1/h$.
Since the low modes essentially represent the large spatial scales
of the solution, the classical data assimilation algorithm works
well for this type of observations. In addition, it is  worth mentioning that this method is  consistent with some of the signal synchronization algorithms. Most recently, similar idea has also been introduced in \cite{FoJoKrTi} to show  that the long-time dynamics of the two-dimensional Navier--Stokes equations  can be imbedded in an infinite-dimensional dynamical system that is induced by  an ordinary differential equations, named {\it determining form}, which is governed by a globally Lipschitz vector field.

The method of constructing $v$, given by (\ref{approxeqn}),
allows the use of general interpolant observables,
given by interpolants $I_h\colon H^1(\Omega)\to L^2(\Omega)$ that are
linear and satisfy the following approximation property:
\begin{equation}
	\|\varphi-I_h(\varphi)\|_{L^2(\Omega)}^2
	\le c_0 h^2 \|\varphi\|_{H^1(\Omega)}^2
		\label{interpa}
\end{equation}
for every $\varphi\in H^1(\Omega)$.
The orthogonal projection onto the low Fourier modes, with
wave numbers k such that $|k|\le 1/h$, mentioned above, is
an example of such interpolant observable. However, there are many other
interpolant observables which satisfy (\ref{interpa}).

One physically relevant example of an interpolant which
satisfies condition (\ref{interpa})
are the volume elements studied
in \cite{jones1992} and \cite{jones1993} (see also
Foias and Titi \cite{FTi}).
In this case
$$
	I_h\big(\varphi(x)\big) =  \sum_{j=1}^N \bar \varphi_j \chi_{Q_j}(x)
\quad\hbox{where}\quad
	\bar \varphi_j = {N\over L^2}
		\int_{Q_j} \varphi(x)\,dx,
$$
and the domain $\Omega=[0,L]^2$, for the periodic boundary conditions case, has been divided into $N$ equal
squares $Q_j,$ with sides $h=L/\sqrt N$.
Volume elements generalize to any domain $\Omega$ on which
the Bramble--Hilbert lemma holds.
An elementary discussion of this lemma
in the context of finite element methods appears in Brenner and
Scott \cite{brennerscott2007}.

In addition, we also consider
interpolant observables given by linear interpolants
$I_h\colon H^2(\Omega)\to L^2(\Omega),$ that satisfy the following approximation property:
\begin{equation}
	\|\varphi-I_h(\varphi)\|_{L^2(\Omega)}^2
	\le {1\over 4} c_0^2h^4 \|\varphi\|_{H^2(\Omega)}^2,
		\label{interpah2}
\end{equation}
for every $\varphi\in H^2(\Omega)$. An example of this type of interpolant is given by measurements
at a discrete set of nodal points in $\Omega$.
% such that every
%point of $\Omega$ lies within a distance $h$ from at least one
%of the $x_i$.
Specifically, let $h>0$ be given, and let $\Omega=\cup_{j=1}^{N_h} Q_j,$ where $Q_j$ are
disjoint subsets such that ${\rm diam}\,Q_j\le h$, for $j=1,2,\ldots,N_h,$
and let $x_j\in Q_j$ be arbitrary points.
Then set, for example,
\begin{equation}
\label{interpnode}
	I_h\big(\varphi(x)\big)=\sum_{k=1}^{N_h} \varphi(x_k)\chi_{Q_j}(x).
\end{equation}
Following ideas in \cite{jones1993} (see also \cite{FT2}) one can show that $I_h(\varphi)$
satisfies (\ref{interpah2}).

Our paper is organized as follows.
First, we recall the functional setting of the two-dimensional
Navier--Stokes equations necessary for our analysis and then
use this setting to formulate our new method of continuous data
assimilation.  After this we proceed to our main task, that of
finding conditions under which the approximate solution obtained
by this algorithm of data assimilation converges to the reference
solution over time.
Section 3
treats the case of smooth, bounded domains
with no-slip Dirichlet boundary conditions, while
section 4 treats
the case of periodic boundary conditions.
Our main results may be stated as follows:

\begin{theorem}
\label{mainres}
Let $\Omega$ be an open, bounded and connected set in $\R^2$
with $C^2$ boundary, and let  $u$ be a solution to equations
(\ref{nseq})--(\ref{incomp}) with
no-slip Dirichlet boundary conditions.
Assume that $I_h$ satisfies (\ref{interpa}), with $h$ small enough such that
$$1/h^2\ge c_1 \lambda_1 G^2,$$
where $c_1$ is a constant given in (\ref{defc1}). Then there exists $\mu>0$, given explicitly in Proposition \ref{mainlem}, such that
$\|v-u\|_{L^2(\Omega)}\to 0$ exponentially, as $t\to\infty$.
\end{theorem}

\noindent
Here $G$ denotes the Grashof number
\begin{equation}
	G={1 \over \nu^2 \lambda_1}
		\limsup_{t\to\infty} \|f(t)\|_{L^2}
\label{grash}
\end{equation}
where $\lambda_1$ is the smallest eigenvalue of the Stokes
operator subject to homogeneous Dirichlet boundary conditions. Let us remark, again, that the constant $c_1$ is   depends only on $c_0$, given in (\ref{interpa}),
and the shape, but not the size, of the domain $\Omega$.
In particular, $c_1$ is given by (\ref{defc1})
where the constant $c$ is chosen
so the bound (\ref{lady}) on the non-linear term holds.
Moreover, $\mu$ may be chosen equal to $5c^2G^2\nu\lambda_1$
as indicated in Proposition \ref{mainlem}, below.

Results similar to Theorem \ref{mainresp} hold when $I_h$ satisfies
(\ref{interpah2}), however, we omit the proof of this result in the
case of no-slip Dirichlet boundary conditions and instead proceed
directly to the the case of periodic boundary conditions where sharper
estimates may be obtained.  In particular, we prove

\begin{theorem}
\label{mainresp}
Let $\Omega=[0,L]^2$ and let $u$ be a solution to equations
(\ref{nseq})--(\ref{incomp}) with periodic boundary conditions.
Let $I_h$ satisfy either (\ref{interpa})
or (\ref{interpah2}), with $h$ small enough such that
$$1/h^2\ge c_2 \lambda_1 G\big(1+ \log (1+G)\big),$$
where $c_2$ is a constant given in (\ref{defc2}). Then
there exists $\mu>0$, given explicitly in Proposition \ref{mainlemp}, such that
	$\|v-u\|_{H^1(\Omega)}\to 0$ exponentially, as $t\to\infty$.
\end{theorem}

\noindent
Let us remark again that $c_2$ depends only on $c_0$, and that
$\mu$ may be chosen as $3c_2\nu\lambda_1 G(1+\log(1+G))/c_0$.
In particular, $\mu$ is given in Proposition \ref{mainlemp}
and
$c_2$ is defined in (\ref{defc2}) as an
increasing function of $c_0$ and $c$, where $c$ is chosen
large enough so that the bounds in both
(\ref{auball}) and (\ref{brezis}) hold.

Note that the estimate on the length scale $h$ in
Theorem~\ref{mainresp} can be compared to previous results
reported in \cite{olson2003}.
Let $\tilde v(t)$ be the approximate solution obtained by
the method of continuous data assimilation introduced
in \cite{olson2003} for
the interpolant observable $I_h(u)$ given by
projection onto the Fourier modes with wave numbers $|k|<1/h$.
In \cite{olson2003}
it was shown, that
for small values of $h$, such that $1/h^2 \sim \lambda_1 G$,
$\|u(t)-\tilde v(t)\|_{H^1(\Omega)}\to 0$ exponentially
fast, as $t\to\infty$.
Up to a logarithmic correction term,
Theorem~\ref{mainresp} states similar estimates on $h$
for the new algorithm which covers
a much wider class of interpolant observables.

The final section of this paper discusses numerical simulations,
which are in progress, related works, and closes with a few
concluding remarks.

\section{Preliminaries}

This section reviews the functional setting of the
two-dimensional Navier--Stokes equations with no-slip and
periodic boundary conditions, recalls some facts that will be
used in the remainder of the paper and then gives an explicit
formulation of our new method for continuous data assimilation
in this context.
Following
Constantin and Foias~\cite{constantin1988},
Foias, Manley, Rosa and Temam~\cite{foias2001},
Robinson~\cite{robinson2001} and
Temam~\cite{temam1983}, we
begin by defining a suitable domain $\Omega$ and space ${\cal V}$
of smooth functions which satisfy each type of boundary conditions.

\medskip
{\bf No-slip Dirichlet Boundary Conditions.}
Let $\Omega$ be an
open, bounded and connected domain with $C^2$ boundary.
%with twice continuously
%differentiable boundary.
Define $\cal V$ to be set of all $C^\infty$ compactly supported
vector fields from $\Omega$ to $\R^2$ that are
divergence free.

\medskip
{\bf Periodic Boundary Conditions.}
Let $\Omega=[0,L]^2$ for some fixed $L>0$.
Define ${\cal V}$ to be
the set of all $L$-periodic trigonometric polynomials
from $\R^2$ to $\R^2$ that are
divergence free and have zero averages.
\medskip

Given ${\cal V}$ corresponding to either type of boundary
conditions let $H$ be the closure of ${\cal V}$ in $L^2(\Omega;\R^2)$
and $V$ be the closure of ${\cal V}$ in $H^1(\Omega;\R^2)$.
The spaces $H$ and $V$ are Hilbert spaces with inner products
$$
	(u,v)=\int_\Omega u(x)\cdot v(x)\,dx\qquad\hbox{and}\qquad
	(\!(u,v)\!)=\sum_{i,j=1}^2 \int_\Omega
		{\partial u_i\over\partial x_j}
		{\partial v_i\over\partial x_j}\,dx,
$$
respectively.
Denote the norms of $H$ and $V$ by
$$	|u|=\sqrt{(u,u)}
		\qquad\hbox{and}\qquad
	\|u\|=\sqrt{(\!(u,u)\!)},$$
and the dual of $V$ by $V^*$ with the pairing $\langle u,v\rangle$
where $u\in V^*$ and $v\in V$.

Define the Leray projector $P_\sigma$ as the orthogonal projection
from $L^2(\Omega;\R^2)$
onto $H$, and define the Stokes operator $A\colon V\to V^*$, and
the bilinear term $B\colon V\times V\to V^*$
to be the continuous extensions of the operators given by
$$
	Au=-P_\sigma \Delta u\qquad\hbox{and}\qquad
	B(u,v)=P_\sigma (u\cdot\nabla v),
$$
respectively, for any smooth solenoidal
vector fields $u$ and $v$ in ${\cal V}$.

Denote the domain of $A$ by ${\cal D}(A)=\big\{\, u\in V : Au\in H\,\big\}$.
The linear operator $A$ is self-adjoint and
positive definite with compact inverse $A^{-1}\colon H\to H$.
Thus, there exists a complete
orthonormal set of eigenfunction $w_i$ in $H$ such that
$Aw_i=\lambda_i w_i$ where $0<\lambda_i\le\lambda_{i+1}$ for $i\in\N$.
Writing $\lambda_1$ as the smallest eigenvalue of $A$
we have the
following Poincar\'e inequalities:
\begin{equation}
	\hbox{if } u\in V\hbox{ then }
	\lambda_1 |u|^2\le \|u\|^2,
	\label{poinc}
\end{equation}
\begin{equation}
	\hbox{if } u\in {\cal D}(A)\hbox{ then }
	\lambda_1 \|u\|^2\le |Au|^2.
	\label{poincA}
\end{equation}
Note that for $u\in H$, $|u|=\|u\|_{L^2(\Omega)}$ and for
$u\in V$ the Poincar\'e inequality implies
$\|u\|$ is equivalent to $\|u\|_{H^1(\Omega)}$.

The bilinear term $B$ has the algebraic property that
\begin{equation}
	\big\langle B(u,v),w\big\rangle = -\big\langle B(u,w),v\big\rangle
	\label{balg}
\end{equation}
for $u,v,w\in V$, and consequently the orthogonality property that
\begin{equation}
	\big\langle B(u,w),w\big\rangle =0.
	\label{borth}
\end{equation}
Here the pairing $\langle \cdot,\cdot\rangle$ denotes the
dual action of $V^*$ on $V$.  Details may be found, e.g., in
\cite{constantin1988}, \cite{foias2001}, \cite{robinson2001}
and \cite{temam1983}.

In the case of periodic boundary conditions the bilinear term
possesses the additional orthogonality property
\begin{equation}
	\big(B(w,w),Aw\big)=0,
\qquad\hbox{for every}\qquad w\in {\cal D}(A);
	\label{balgp}
\end{equation}
and consequently one has
\begin{equation}
	\big(B(u,w),Aw\big)
	+\big(B(w,u),Aw\big)
	=-\big(B(w,w),Au\big).
	\label{borthp}
\end{equation}

Note that the bilinear term satisfies a number of inequalities
which hold for either no-slip or periodic boundary conditions.
These are
\begin{equation}
	\big|\big\langle B(u,v),w\big\rangle \big|\le
		c |u|^{1/2}\|u\|^{1/2} \|v\| |w|^{1/2} \|w\|^{1/2},
	\label{lady}
\end{equation}
for every $u,v,w\in V$,
\begin{equation}
	\big|\big(B(u,v),w\big)\big|\le
		c |u|^{1/2}\|u\|^{1/2} \|v\|^{1/2} |Av|^{1/2} |w|
	\label{beqmf}
\end{equation}
for every $u\in V$, $v\in{\cal D}(A)$ and $w\in H$,
and
\begin{equation}
	\big|\big(B(u,v),w\big)\big|\le
		c |u|^{1/2}|Au|^{1/2} \|v\| |w|,
	\label{beqlf}
\end{equation}
for every $u\in {\cal D}(A)$ and $v,w\in V$,
where $c$ is a dimensionless constant depending only on
the shape, but not the size, of $\Omega$.
These inequalities may be obtained from the H\"older's inequality,
the Sobolev inequalities and Ladyzhenskaya's inequality, see, e.g.,
\cite{constantin1988}, \cite{foias2001}, \cite{robinson2001}
and \cite{temam1983}.

%Note there are two other inequalities we shall use which
%contain a dimensionless constant depending only on $\Omega$.
%These are (\ref{auball}) and (\ref{brezis}) which will be
%used for our results in the periodic case.  For simplicity
%we denote the constant appearing in each of these
%inequalities by $c$ with the understanding that $c$ is
%sufficiently large that all the required bounds hold.

With the above notation we write the incompressible
two-dimensional Navier--Stokes equations in functional form as
\begin{equation}
	{du\over dt}+\nu Au+B(u,u)=f
		\label{nseqfunc}
\end{equation}
with initial condition $u(0)=u_0$.
We have assumed $f\in H$ so that $P_\sigma f=f$.
As shown in
\cite{constantin1988}, \cite{foias2001}, \cite{robinson2001}
and \cite{temam1983}
these equations are well-posed; and possess a
compact finite-dimensional global attractor, when $f$ is time-independent.
Specifically, we have

\begin{theorem} [Existence and Uniqueness of Strong Solutions]
\label{exuniq}
Suppose $u_0\in V$ and $f\in L^\infty\big((0,\infty),H\big)$.  Then the initial value problem
(\ref{nseqfunc}) has a unique solution that satisfies
$$
	u\in C\big([0,T];V\big)\cap L^2\big((0,T);D(A)\big)
\quad\hbox{and}\quad
	{du\over dt}\in L^2\big((0,T);H\big),
$$
for any $T>0$.
\end{theorem}
\noindent

We now give bounds on solutions $u$ of
(\ref{nseqfunc}) that will be used in our later analysis.
With the exception of inequality (\ref{auball}) due to
Dascaliuc, Foias and Jolly \cite{dascaliuc2010} these
estimates appear in any the references listed above.

\begin{theorem}
\label{globalb}
Fix $T>0$, and let $G$ be the Grashof number given in (\ref{grash}). Suppose that $u$ is the solution of (\ref{nseqfunc}), corresponding to the initial value $u_0$, then there
 exists a time $t_0$, which depends on $u_0$, such that for all $t\ge t_0$ we have:
\begin{equation}
	|u(t)|^2\le 2\nu^2 G^2 \qquad\hbox{and}\qquad
	\int_{t}^{t+T} \|u(\tau)\|^2 d\tau
		\le 2\big(1 + T\nu\lambda_1\big) \nu G^2.
	\label{avgv}
\end{equation}
In the case of periodic boundary conditions we also have:
\begin{equation}
	\|u(t)\|^2 \le 2\nu^2\lambda_1 G^2,\qquad
	\int_{t}^{t+T} |Au(\tau)|^2 d\tau
	\le 2(1+T\nu\lambda_1)\nu\lambda_1 G^2;
	\label{avgda}
\end{equation}
furthermore, if $f\in H$ is time-independent then
\begin{equation}
	|Au(t)|^2 \le c\nu^2\lambda_1^2 (1+G)^4.
	\label{auball}
\end{equation}
\end{theorem}

\noindent
%Note in the limit as $t\to \infty$ the bounds
%appearing in Theorem \ref{globalb} may each be reduced by half.
%In particular, $t_0$ can be chosen large enough such that the
%factor 2 on the right side of each of the above inequalities
%becomes a term of the form $1+\epsilon$ where $\epsilon>0$.
%We shall not pursue such improvements in the estimates that
%follow.

We now write the continuous data assimilation equations
(\ref{approxeqn}) for the incompressible two-dimensional
Navier--Stokes
equations.
%(\ref{nseq})--(\ref{incomp}).
Let $u$ be a strong solution of (\ref{nseq})--(\ref{incomp}),
or equivalently (\ref{nseqfunc}), as given
by Theorem \ref{exuniq}, and let $I_h$ be an interpolation
operator satisfying~(\ref{interpa}) or (\ref{interpah2}).
Suppose that $u$ is to be recovered from the observational
measurements $I_h(u(t))$, that have been
continuously recorded for times $t$ in $[0,T]$.
Then, the approximating solution $v$ with initial condition
$v_0\in V$, chosen arbitrarily, shall be given by

\[ {\partial v\over \partial t} -\nu \Delta v
	+(v\cdot \nabla)v+\nabla q = f+\mu (I_h(u) -I_h(v)),
\]
\[
	\nabla\cdot v = 0,
\]
on the interval $[0,T]$.
Using the above functional setting the above system is equivalent to
\begin{equation}
	{dv\over dt}+\nu Av+B(v,v)=f+\mu P_\sigma(I_h(u) -I_h(v)),
	\label{vapprox}
\end{equation}
on the interval $[0,T]$.

If we knew $u_0$ exactly, then we could
take $v_0=u_0$ and the resulting solution $v$ would be
identical to $u$ for all time; this is due to the uniqueness of the solutions of (\ref{vapprox}) (see Theorem \ref{vwposed}, below).
However, if we knew $u_0$ exactly, there would be no
need for continuous data assimilation in the first place
and one could integrate (\ref{nseqfunc}) directly
with the initial value $u_0$.
Intuitively speaking it makes sense to take
$v_0=P_\sigma I_h(u(0))$, which is the initial
observation of the solution $u$.
However, $v_0$ chosen in this way may not be an element of $V$.
The main point of the data assimilation method given
in (\ref{vapprox}) is to avoid the difficulties which come from
the direct insertion of observational measurements into the
approximate solution.
A choice for $v_0$ in agreement with this philosophy is $v_0=0$.
In fact, our results to hold equally well
when $v_0$ is chosen to be any element of $V$.
In either case we obtain an approximating solution $v$
constructed using only the observations of
the solution $I_h(u)$ and the known values of $\nu$ and $f$.

We now show the data assimilation equations (\ref{vapprox}) are well-posed.
When $I_h$ satisfies (\ref{interpa}) we show well-posedness for both
no-slip Dirichlet and periodic boundary conditions.
When $I_h$ satisfies (\ref{interpah2})
we will deal here, for simplicity, with only the case of periodic
boundary conditions.

\begin{theorem}
\label{vwposed}
Suppose $I_h$ satisfies (\ref{interpa}) and
$\mu c_0h^2\le\nu$, where $c_0$ is the constant appearing
in (\ref{interpa}).
Then the continuous data assimilation equations (\ref{vapprox})
possess unique strong solutions that satisfy
\begin{equation}
	v\in C\big([0,T];V\big)\cap L^2\big((0,T);D(A)\big)
\quad\hbox{and}\quad
	{dv\over dt}\in L^2\big((0,T);H\big),
	\label{apreg}
\end{equation}
for any $T>0$.
Furthermore, this solution
depends continuously on the initial data $v_0$ in the $V$ norm.
\end{theorem}

\begin{proof}
Define $g=f+\mu P_\sigma I_h(u)$.  Theorem \ref{exuniq}
implies $u\in C\big([0,T];V\big)$.  Consequently
$$|P_\sigma I_h(u)|\le |u-I_h(u)|+|u|
	\le \big(c_0^{1/2}h +\lambda_1^{-1/2}\big)\|u\|
$$
implies that $P_\sigma I_h(u)\in C\big([0,T];H\big)$.
Hence $g\in C\big([0,T];H\big)$.
This means there is a constant $M$ such that
$|g|^2<M$ for every $t\in [0,T]$.

We now show the existence of solutions $v$ to
(\ref{vapprox}) using the Galerkin method.
The proof follows the same ideas as the proof of
Theorem \ref{exuniq}.
Let $P_n$ be the
$n$-th Galerkin projector and $v^n$ be the solution to
the finite-dimensional Galerkin truncation
\begin{plain}
\begin{equation}
	\cases{\displaystyle
	{dv^n\over dt}+\nu Av^n+P_nB(v^n,v^n)=P_n g- \mu P_n I_h(v^n)\cr\cr
	v^n(0)=P_n v_0.\cr
	}
	\label{galerkinv}
\end{equation}
\end{plain}%
First, we observe that (\ref{galerkinv}) is a finite system of
ODEs, which has short time existence and uniqueness.
We focus on the maximal interval of existence, $[0,T_n)$, and
show uniform bound for   $v_n$, which are independent of $n$. This in turn will imply the global existence for  (\ref{galerkinv}).
Thus, our aim is to find bounds on $v^n$ which are uniform in $n$.
This will
then show global existence of solutions to (\ref{vapprox}).
In the estimates that follow, we denote the Galerkin solution $v^n$
by $v$ for notational simplicity.

Begin by taking inner products of (\ref{galerkinv}) with $v$ to
obtain
\begin{plain}
$$\eqalign{
	{1\over 2} {d\over dt} |v|^2+\nu\|v\|^2
		&=(g,v)-\mu\big(I_h(v),v\big)\cr
		&=(g,v)+\mu\big(v-I_h(v),v\big)-\mu|v|^2\cr
		&\le{1\over 2\mu}|g|^2+{\mu\over 2}|v|^2
			+{\mu\over 2} \big|P_\sigma (v-I_h(v))\big|^2
			+{\mu\over 2}|v|^2-\mu|v|^2\cr
		&\le{1\over 2\mu}|g|^2
			+{\mu c_0h^2\over 2} \|v\|^2.\cr
}$$
\end{plain}%
By hypothesis $h$ is so small that $\mu c_0h^2\le \nu.$
Therefore,
\begin{equation}
	{d\over dt} |v|^2+\nu\|v\|^2\le {1\over \mu}|g|^2,
	\label{svineq}
\end{equation}
and consequently
\begin{equation}
	{d\over dt} |v|^2+\nu\lambda_1 |v|^2\le {1\over \mu}M,
\quad\hbox{for every}\quad
	t\in [0,T_n).
	\label{rhineq}
\end{equation}
Multiplying (\ref{rhineq}) by $e^{\nu\lambda_1 t}$ and integrating
yields
$$
	|v(t)|^2\le |v_0|^2 e^{-\nu\lambda_1 t}
		+ {M\over\mu\nu\lambda_1}\Big(1-e^{-\nu\lambda_1 t}\Big)
		\le \rho_H^2,
\quad\hbox{for every}\quad
	t\in [0,T_n),
$$
where $$\rho_H^2=|v_0|^2+ {M\over\mu\nu\lambda_1}.$$
As this bound holds uniformly in $n$ for $T_n$ arbitrarily large,
we have global existence on the interval $[0,T]$, for all $T\ge 0$.
Now, integrating (\ref{svineq}) yields
$$
	|v(t)|^2-|v_0|^2+\nu\int_0^t \|v\|^2\le {t\over \mu}M.
$$
It follows that
$$
	\int_0^t\|v(\tau)\| d\tau\le \sigma_V^2,
\quad\hbox{for every}\quad t\in[0,T],
$$
where
$$\sigma_V^2={1\over\nu}|v_0|^2+{T\over\mu\nu} M.$$

Now, take inner products of (\ref{galerkinv}) with $Av$ to obtain
$$
{1\over 2}{d\over dt} \|v\|^2+\nu|Av|^2+\big(B(v,v),Av\big)
	=(g,Av)-\mu(I_h(v),Av).$$
Inequality (\ref{beqmf}) implies
\begin{plain}
$$\eqalign{
	\big|\big(B(v,v),Av\big)\big|&\le
	c|v|^{1/2}\|v\||Av|^{3/2}\cr
	&\le {1\over 4}
		\left({6^{3/4}\over\nu^{3/4}}c |v|^{1/2}\|v\|\right)^{4}
	+{3\over 4}
		\left({\nu^{3/4}\over 6^{3/4}}|Av|^{3/2}\right)^{4/3}\cr
	&\le {54c^4\over \nu^3} |v|^2\|v\|^4 + {\nu\over 8}|Av|^2.
}$$
\end{plain}%
Furthermore,
$$
	\big|(g,Av)\big|\le |g||Av|
		\le {2\over\nu}|g|^2 +{\nu\over 8}|Av|^2
$$
and by (\ref{interpa}) along with the assumption that $\mu c_0h^2\le\nu$
we obtain
\begin{plain}
$$\eqalign{
	\mu\big|(I_h(v),Av)\big|
		&\le {\mu^2\over\nu}\big|v-I_h(v)\big|^2
			+{\nu\over 4}|Av|^2-\mu\|v\|^2\cr
		&\le {\mu^2c_0h^2\over\nu}\|v\|^2
			+{\nu\over 4}|Av|^2-\mu\|v\|^2
			\le{\nu\over 4}|Av|^2.
}$$
\end{plain}%
Therefore,
\begin{equation}
	{d\over dt}\|v\|^2+\nu|Av|^2\le
		{54c^4\over \nu^3}|v|^2\|v\|^4+{4\over \nu}|g|^2,
	\label{sdineq}
\end{equation}
and consequently
\begin{equation}
	{d\over dt}\|v\|^2-{54c^4\over \nu^3}|v|^2\|v\|^4
	\le {4\over \nu}|g|^2
	\le {4\over \nu}M,
	\label{rvineq}
\end{equation}
for every $t\in[0,T]$.  Define
\begin{equation}
	\psi(t)=\exp\bigg\{-{54c^4\over\nu^3}\int_0^t |v|^2\|v\|^2\bigg\}.
	\label{intfact}
\end{equation}
Since
$$
	\int_0^t |v|^2\|v\|^2
		\le \rho_H^2\int_0^t\|v\|^2\le \rho_H^2\sigma_V^2<\infty,
\quad\hbox{for every}\quad t\in[0,T],
$$
we have that $\psi(t)>0$ for every $t\in[0,T]$.
Multiplying (\ref{rhineq}) by $\psi(t)$ and integrating yields
$$
	\|v(t)\|^2\le {1\over \psi(t)}
		\bigg\{\|v_0\|^2+ {4\over\nu}M\int_0^t\psi(s)ds\bigg\}
	\le \rho_V^2,
\quad\hbox{for all}\quad
	t\in[0,T],
$$
where
$$
	\rho_V^2={1\over\psi(T)}
		\bigg\{\|v_0\|^2+ {4T\over\nu}M\bigg\}.
$$
Now, integrating (\ref{sdineq}) yields
$$
	\|v(t)\|^2-\|v_0\|^2+\nu\int_0^t|Av|^2
		\le {54c^4\over\nu^3}\int_0^t
			\Big(|v|^2\|v\|^4+ {4\over\nu} |g|^2\Big)
		\le\sigma_{{\cal D}(A)}^2,
$$
for every $t\in[0,T]$, where
$$
\sigma_{{\cal D}(A)}^2=
 {54c^4T\over\nu^3}\
			\bigg\{\rho_H^2\rho_V^4+ {4\over\nu} M\bigg\}.
$$

The bounds $\rho_V$ and $\sigma_{{\cal D}(A)}$ are uniform in $n$.
Uniform bounds on $|dv/dt|$ then proceed in exactly the same
way as for the two-dimensional Navier--Stokes equations.
Since the estimates on the Galerkin solutions are uniform in $n$,
Aubin's compactness theorem \cite{aubin1963} allows one to extract
subsequences in such a way that the limit $v$ satisfies
(\ref{vapprox}) and (\ref{apreg}).

Next, we show that such solutions are unique and depend
continuously on the initial data.
Let $v_1$ and $v_2$ be two solutions for (\ref{vapprox})
both satisfying the conditions in (\ref{apreg}).
Choose $K$ large enough such that
$\|v_1\|^2\le K$ and $\|v_2\|^2\le K$ for almost every
$t\in[0,T]$.  Let $\delta=v_1-v_2$.  Then $\delta$ satisfies
$$
	{d\delta\over dt}+\nu A \delta
		+B(v_1,\delta)+B(\delta,v_2)
	=-\mu P_\sigma I_h(\delta).
$$
Taking inner product with $A\delta$ yields
$$
	{1\over 2}{d\over dt}\|\delta\|^2+\nu|A\delta|^2
		+\big(B(v_1,\delta),A\delta\big)
		+\big(B(\delta,v_2),A\delta\big)
	=-\mu(I_h(\delta),A\delta).
$$
Here we used the fact that
$$
{1\over 2}{d\over dt} \|\delta\|^2=\Big({d\delta\over dt}, A \delta\Big),
$$
which can be justified by Lemma 1.2 in Chapter 3 of Temam \cite{temam2001}
or Theorem 7.2 in Robinson \cite{robinson2001}
which is due to Lions--Magenes \cite{lions1972}.
Estimate the right-hand side of this equation as
\begin{plain}
$$\eqalign{
	-\mu(I_h(\delta),A\delta)
	&=\mu(\delta-I_h(\delta),A\delta)-\mu\|\delta\|^2\cr
	&\le {\mu^2\over 2\nu}\big|P_\sigma(\delta-I_h(\delta))
		\big|^2+{\nu\over 2}|A\delta|^2-\mu\|\delta\|^2\cr
	&\le {\mu^2c_0h^2\over 2\nu}\|\delta\|^2
		+{\nu\over 2}|A\delta|^2-\mu\|\delta\|^2
	\le {\nu\over 2}|A\delta|^2,
}$$
\end{plain}%
where we have again used the hypothesis that $\mu c_0h^2\le\nu$.
It follows that
\begin{equation}
	{1\over 2}{d\over dt}\|\delta\|^2
		+{\nu\over 2}|A\delta|^2
	\le
		\big|(B(v_1,\delta),A\delta)\big|
		+\big|(B(\delta,v_2),A\delta)\big|.
	\label{deltaeq}
\end{equation}
\par
The proof of uniqueness and continuity now proceeds as
for the two-dimensional Navier--Stokes equations.
In particular, estimate the non-linear terms on the right-hand
side of (\ref{deltaeq}) using
(\ref{beqmf}) and (\ref{beqlf}) as
\begin{plain}
$$\eqalign{
	\big|(B(v_1,\delta),A\delta)\big|
	&\le
		c|v_1|^{1/2}\|v_1\|^{1/2}\|\delta\|^{1/2}|A\delta|^{3/2}
	\le
		{27c^4K^2\over 4\nu^3\lambda_1}\|\delta\|^2+ {\nu\over 4}|A\delta|^2,
}$$
\end{plain}%
and
\begin{plain}
$$\eqalign{
	\big|(B(\delta,v_2),A\delta)\big|
	&\le
		c |\delta|^{1/2} \|v_2\| |A\delta|^{3/2}
	\le
		{27c^4K^2\over 4\nu^3\lambda_1}\|\delta\|^2
		+ {\nu\over 4}|A\delta|^2.
}$$
\end{plain}%
Therefore,
$$
	{d\over dt}\|\delta\|^2
	\le {27c^4K^2\over 2\nu^3\lambda_1} \|\delta\|^2,
\quad\hbox{for all}\quad t\in[0,T].
$$
Integrating yields
$$
	\|\delta(t)\|^2\le \|\delta_0\|^2
		\exp\bigg\{{27c^4K^2\over 2\nu^3\lambda_1} t\bigg\}.
$$
\par
Thus, the solutions $v$ to (\ref{vapprox}), which satisfy (\ref{apreg}),
also satisfy $v\in{\cal C}([0,T],V)$, and depend continuously on
the initial data in the $V$ norm.
\end{proof}

\begin{theorem}
\label{vwposedh2}
In the case of periodic boundary conditions
suppose that $I_h$ satisfies (\ref{interpah2}), and
$\mu c_0h^2\le\nu$, where $c_0$ is the constant appearing
in (\ref{interpah2}).
Then the continuous data assimilation equations (\ref{vapprox})
possess unique strong solutions that satisfy (\ref{apreg}),
%\begin{equation}
%	v\in C\big([0,T];V\big)\cap L^2\big((0,T);D(A)\big)
%\quad\hbox{and}\quad
%	{dv\over dt}\in L^2\big((0,T);H\big)
%	\label{apregh2}
%\end{equation}
for any $T>0$.
Furthermore, this solution is in $C\big([0,T],V\big)$ and
depends continuously on the initial data $v_0$ in the $V$ norm.
\end{theorem}
\begin{proof}
The proof
is similar to the proof of Theorem \ref{vwposed} but
makes use of the identity (\ref{balgp}) to obtain
estimates on $\|v\|$ and $\int_0^t |Av|^2$ directly.
\end{proof}

The algorithm
given by equation (\ref{vapprox})
for constructing
the approximate solution $v$
contains two parameters $h$ and $\mu$.
The first parameter $h$ has dimensions of length and
corresponds to the resolution of the observational
measurements represented by $I_h(u)$.
Smaller values of $h$ correspond to spatially more accurate
resolved measurements.
The relaxation parameter $\mu$ controls the rate at which the
approximating solution~$v$ is forced toward the
observable part of the reference solution $u$.
Larger values of $\mu$ cause $I_h(v)$ to faster track $I_h(u)$.
It is the parameter $\mu$ which distinguishes (\ref{vapprox}) from
the previous methods of continuous data assimilation studied
in \cite{browning1998}, \cite{henshaw2003}, \cite{olson2003}
and \cite{olson2008}.

The condition that $\mu c_0h^2\le \nu$, given in Theorem \ref{vwposed},
places a restriction on the size of $\mu h^2$ compared
to the viscosity $\nu$, sufficient to ensure the data assimilation
equations are well-posed.
This restriction is due to the fact that the
%term $\mu I_h(v)$, i.e.
the interpolant operator $\mu I_h$ might generate large gradients and
spatial oscillations (``spill over" to the fine scales) that need
to be controlled (suppressed) by the viscosity term.
Notice that in the case
when $I_h=P_{m_h}$, where $P_{m_h}$ is the
orthogonal projection onto the linear sub-space spanned by the Fourier
modes with wave numbers $|k|\le m_h = \frac{1}{h}$, such oscillations
are not generated, since $-(\mu I_h(v), v) =-\mu|P_{m_h} v|^2$ and
$-(\mu I_h(v), Av) =-\mu\|P_{m_h} v\|^2$.
Consequently, there is
no restriction on $\mu h^2$ and
$\mu$ can be taken arbitrary large.
In the limit when $\mu \to \infty$ one obtains
exactly the same algorithm introduced in \cite{olson2003} (see also
\cite{olson2011}).
In particular, one has $P_{m_h} v=P_{m_h} u$, and all that one needs to
do is to solve for $q=(I-P_{m_h}) v$, for which an explicit evolution
equation is presented in \cite{olson2003}.

Next, our aim is to give further conditions on $h$ and $\mu$ which
guarantee that the difference between the approximating solution $v$
and the reference solution $u$ converges to zero as $t\to\infty$.
To do this we consider the time evolution of $w=u-v$.  Since $$
	B(u,u)-B(v,v)=B(u,w)+B(w,v)=B(u,w)+B(w,u)-B(w,w)$$
and $$
	I_h(u)-I_h(v)=I_h(w)
$$ then subtracting equation (\ref{vapprox}) from equation
(\ref{nseqfunc}) yields \begin{equation}
	{dw\over dt}+\nu A w + B(u,w)+B(w,u)-B(w,w)
		=-\mu P_\sigma I_h(w).
	\label{weq}
\end{equation} This equation serves as the starting point for the proofs
of Theorem \ref{mainres} and Theorem \ref{mainresp} in the proceeding
two sections.

\section{No-slip Dirichlet Boundary Conditions Case}

In this section we prove Theorem \ref{mainres}.
We first recall the following generalized Gronwall inequality
proved in Jones and Titi \cite{jones1992}, see also \cite{FTi}.

\begin{lemma}[Uniform Gronwall Inequality]
Let $T>0$ be fixed.  Suppose
$$
	{dY\over dt}+ \alpha(t) Y\le0,
\qquad\hbox{where}\qquad
	\limsup_{t\to\infty}\int_t^{t+T} \alpha(s)ds\ge \gamma>0.
$$
Then $Y(t)\to 0$ exponentially, as $t\to\infty$.
\label{unigron}
\end{lemma}

We now state and prove a lemma leading to our main result.

\begin{proposition}
Let $\Omega$ be an open, bounded and connected set in $\R^2$
with $C^2$ boundary, and
let $u$ be a solution of the incompressible two-dimensional
Navier--Stokes equations (\ref{nseqfunc})
on $\Omega$ with no-slip Dirichlet boundary conditions.
Let $v$ be the approximating solution given by equations (\ref{vapprox}).
Then $|u-v|\to 0$, as $t\to \infty$,
provided $\mu c_0h^2\le \nu$ and $\mu\ge 5c^2G^2\nu\lambda_1$.
\label{mainlem}
\end{proposition}
\begin{proof}
Let $w=u-v$.
Then $w$ satisfies equation (\ref{weq}) stated above.
Taking the inner product with $w$ we obtain
\begin{plain}
$$\eqalign{
{1\over 2}{d\over dt}|w|^2&
	+\nu\|w\|^2+\big(B(w,u),w\big)=-\mu(I_h(w),w)\cr
	&=\mu(w-I_h(w),w)-\mu|w|^2\cr
	&\le {\mu\over 2}\big|P_\sigma(w-I_h(w))\big|^2
		+{\mu\over 2} |w|^2-\mu|w|^2\cr
	&\le {\mu c_0h^2\over 2}\|w\|^2-{\mu\over 2} |w|^2
	\le {\nu\over 2}\|w\|^2-{\mu\over 2} |w|^2
	.\cr
}$$
\end{plain}%
Since (\ref{lady}) implies
$$
	\big|(B(w,u),w)\big|
		\le c \|u\| |w|\|w\|
		\le {c^2\over 2\nu}\|u\|^2|w|^2+{\nu\over 2}\|w\|^2,
$$
we obtain
$$
	{d\over dt} |w|^2
		+\Big(\mu-{c^2\over\nu}\|u\|^2\Big)|w|^2\le 0.
$$
Denote
$$\alpha(t)=\mu-{c^2\over \nu}\|u\|^2.$$
Taking $T=(\nu\lambda_1)^{-1}$ in Theorem \ref{globalb}
we have for $t\ge t_0$ that
$$
	\int_t^{t+T} \|v\|^2
		\le 2(1+T\nu\lambda_1) \nu G^2=4\nu G^2.
$$
Thus
$$
	\limsup_{t\to\infty} \int_t^{t+T}\alpha(s)ds
		\ge {\mu\over \nu\lambda_1}-4c^2G^2\ge c^2G^2>0,
$$
and by Lemma \ref{unigron} it follows
that $|w|\to 0$, exponentially, as $t\to \infty$.
\end{proof}

\begin{proof}[Proof of Theorem \ref{mainres}]
The hypothesis of Proposition \ref{mainlem} require that
$$
	\mu c_0h^2\le\nu
\qquad\hbox{and}\qquad
	\mu\ge 5c^2G^2\nu\lambda_1.
$$
Therefore
\begin{equation}
	{1\over h^2}\ge {\mu c_0\over\nu} \ge c_1 G^2\lambda_1
	\label{defc1}
\end{equation}
where $c_1=5c_0c^2$.
\end{proof}

\section{Periodic Boundary Conditions Case}

In this section we prove Theorem \ref{mainresp}.
We begin with an elementary inequality which will be
be referred to in the sequel.
\begin{lemma}
\label{minlog}
Let $\phi(r)=r-\beta(1+\log r)$ where $\beta>0$.  Then
$$
	\min\{\phi(r):r\ge 1\}\ge-\beta\log \beta.$$
\end{lemma}
\begin{proof}
Note first that
$$
	\phi(1)=1-\beta
\qquad\hbox{and}\qquad
	\lim_{r\to\infty} \phi(r)=\infty.
$$
The derivative
$\phi'(r)=1-{\beta/r}$ is zero
if and only if
$r=\beta$.  Therefore
\begin{plain}
$$
	\min\{\phi(r):r\ge 1\}=\cases{
		1-\beta & if $0<\beta\le 1$\cr
		-\beta\log\beta&if $1< \beta$.\cr
}$$
Observe that over the interval $0<\beta\le 1$ we have
$1-\beta\ge -\beta\log\beta$, which concludes our proof.
%To see this we consider $\psi(\beta)=1-\beta+\beta\log\beta$.
%Then $\psi(0^+)=1$ and $\psi(1)=0$.
%Since $\psi'(\beta)=\log\beta\le 0$ it follows that $\psi(\beta)\ge 0$
%for $0<\beta\le 1$.
\end{plain}%
\end{proof}

We now state and prove a lemma leading to the
proof of Theorem \ref{mainresp}.

\begin{proposition}
\label{mainlemp}
Let $\Omega=[0,L]^2$, for some fixed $L>0$.
Let $u$ be a solution of the incompressible two-dimensional
Navier--Stokes equations (\ref{nseqfunc})
on $\Omega$ equipped with periodic boundary conditions.
Let $v$ be the approximating solution given by equations (\ref{vapprox}),
where $I_h$ satisfies (\ref{interpa}).
Then $\|u-v\|\to 0$, as $t\to \infty$,
provided $\mu c_0h^2\le \nu$ and
$\mu\ge 3\nu\lambda_1 \big(2c\log 2c^{3/2}+8c\log(1+G)\big)G$.
\end{proposition}

\begin{proof}
The proof makes use of the orthogonality properties
(\ref{balgp}) and (\ref{borthp}) along with the
Br\'ezis--Gallouet inequality \cite{brezis1980}
which may be written as
\begin{equation}
	\|u\|_{L^{\infty}(\Omega)}
		\le c\|u\|\bigg\{
		1+\log {|Au|^2\over \lambda_1 \|u\|^2}\bigg\},
	\label{brezis}	
\end{equation}
which will allow us to obtain sharper estimates than for the case of
no-slip boundary conditions.

Take the inner product of equation (\ref{weq}) with $Aw$ and
and use the orthogonality relations (\ref{balgp}) and
(\ref{borthp}) to obtain
$$
	{1\over 2} {d\|w\|^2\over dt}
	+\nu |Aw|^2 = \big(B(w,w),Au\big)-\mu (I_h(w),Aw).
$$
Using (\ref{brezis}) and the hypothesis $\mu c_0
h^2\le\nu$ we have
$$
	\big|\big(B(w,w),Au\big)\big|
		\le c\|w\|^2
		\bigg\{1+\log {|Aw|^2\over \lambda_1 \|w\|^2}\bigg\}|Au|,
$$
and
\begin{plain}
$$
	\eqalign{
	-\mu (I_h(w), Aw)&=\mu(w-I_h(w),Aw)-\mu\|w\|^2\cr
	&\le \mu|P_\sigma(w-I_h(w))||Aw| -\mu \|w\|^2\cr
	&\le {\mu^2c_0h^2\over 2\nu}\|w\|^2
		+{\nu\over 2}|Aw|^2 -\mu \|w\|^2
	\le {\nu\over 2}|Aw|^2 -{\mu\over 2}\|w\|^2.\cr
	}
$$
\end{plain}%
Therefore,
$$
	{d\|w\|^2\over dt} + \nu|Aw|^2
	\le\bigg( 2c|Au|\bigg\{1+\log{|Aw|^2\over \lambda_1\|w\|^2}\bigg\}
			- \mu\bigg)\|w\|^2,
$$
or
$$
	{d\|w\|^2\over dt} + \bigg(
		\nu\lambda_1 {|Aw|^2\over \lambda_1 \|w\|^2}
	-2c|Au|\bigg\{1+\log{|Aw|^2\over \lambda_1\|w\|^2}\bigg\}
			+ \mu\bigg)\|w\|^2\le 0.
$$

Now setting
$$
	\beta={2c|Au|\over\nu\lambda_1}
\qquad\hbox{and}\qquad
	r={|Aw|^2\over \lambda_1\|w\|^2}
$$
in Lemma \ref{minlog},
and noting that $r\ge 1$, by Poincar\'e's inequality
(\ref{poincA}), we obtain
$$
	{d\|w\|^2\over dt}
	+\bigg\{\mu-2c|Au|\log {2c|Au|\over \nu\lambda_1}
	\bigg\}\|w\|^2\le 0.
$$
By (\ref{auball}) we estimate
$$
	2c\log {2c|Au|\over\nu\lambda_1}
		\le J,
$$
where
\begin{equation}
		J=c_3+c_4\log(1+G),
	\label{jdef}
\end{equation}
$c_3=2c\log 2c^{3/2}$ and $c_4=8c$.
It follows that
$$
	{d\|w\|^2\over dt}
		+\Big\{\mu-J|Au|\Big\}\|w\|^2\le 0.
$$
Furthermore, Young's inequality
$$
	J|Au|\le {J^2\over 2\mu}|Au|^2 + {\mu\over 2}$$
implies
$$
	{d\|w\|^2\over dt}
		+{1\over 2}\bigg\{\mu-
		{J^2\over\mu}|Au|^2\bigg\}\|w\|^2\le 0.
$$
Denote
$$
	\alpha(t)={1\over 2}\bigg\{\mu-{J^2 \over \mu}|Au(t)|^2\bigg\}.
$$
Taking $T=(\nu\lambda_1)^{-1}$ in Theorem \ref{globalb}
we have for $t\ge t_0$ that
$$
	\int_t^{t+T}|Au|^2 \le
		2(1+T\nu\lambda_1)\nu\lambda_1 G^2
		=4\nu\lambda_1 G^2.
$$
Thus,
$$
	\limsup_{t\to\infty}
		\int_t^{t+T}\alpha(s)ds\ge
		{\mu\over2\nu\lambda_1}-{2\nu\lambda_1\over\mu}J^2G^2
	={5\over 6} JG>0,
$$
and consequently  $\|w\|\to 0$ exponentially, as $t\to\infty$.
\end{proof}

\begin{proposition}
\label{mainlemph2}
Let $\Omega=[0,L]^2$, for some fixed $L>0$.
Let $u$ be a solution of the incompressible two-dimensional
Navier--Stokes equations (\ref{nseqfunc})
on $\Omega$, equipped with periodic boundary conditions.
Let $v$ be the approximating solution given by equations (\ref{vapprox})
where $I_h$ satisfies (\ref{interpah2}).
Then $\|u-v\|\to 0$, as $t\to \infty$,
provided $\mu c_0h^2\le \nu$ and
$\mu\ge 3\nu\lambda_1 \big(2c\log 2c^{3/2}+8c\log(1+G)\big)G$.
\end{proposition}
\begin{proof}
The proof is the same as the proof of Proposition \ref{mainlemp} except
that the estimate for $-\mu\big(I_h(w),Aw\big)$ has to be modified as
\begin{plain}
$$
	\eqalign{
	-\mu (I_h(w), Aw)&=\mu(w-I_h(w),Aw)-\mu\|w\|^2\cr
	&\le \mu|w-I_h(w)||Aw| -\mu \|w\|^2\cr
	&\le {\mu^2c_0^2h^4\over 4\nu}|Aw|^2
		+{\nu\over 4}|Aw|^2 -\mu \|w\|^2
	\le {\nu\over 2}|Aw|^2 -\mu\|w\|^2.\cr
	}
$$
\end{plain}%
Then, since $-\mu<-\mu/2$ the rest of the proof follows without change.
\end{proof}

\begin{proof}[Proof of Theorem \ref{mainresp}]
The hypothesis of Proposition \ref{mainlemp} or Proposition \ref{mainlemph2}
require that
$$
	\mu c_0h^2\le \nu
\qquad\hbox{and}\qquad
	\mu\ge 3\nu\lambda_1 JG.
$$
Therefore,
\begin{equation}
	{1\over h^2}\ge {\mu c_0\over \nu}\ge c_2\lambda_1G\big(1+\log(1+G)\big),
	\label{defc2}
\end{equation}
where $c_2=3\max\{c_3,c_4\}$.
\end{proof}

\section{Conclusions}

As shown in this paper, the algorithm given by (\ref{approxeqn}),
for constructing $v(t)$ from the observations
$I_h(u(t))$, yields an approximation for $u(t)$ such that
\begin{equation}
\|u(t)-v(t)\|_{L^2(\Omega)}\to 0
\qquad\hbox{exponentially, as}  t\to\infty,
	\label{c1x}
\end{equation}
provided the observations have fine enough spatial resolution.
This result has the following consequence.
To accurately predict $u(t)$ for time $T$
into the future it is sufficient to have observational
data $I_h(u(t))$ accumulated over an interval of time
linearly proportional to $T$ in the immediate past.

In particular, suppose it is desired to predict $u(t)$ with
accuracy $\epsilon>0$ on the interval $[t_1,t_1+T^*]$,
where $t_1$ is the present time and $T^*>0$
determines how far into the future to
make the prediction.
Let $h$ be small enough and $\mu$ large
enough so that Theorem \ref{mainres} implies
(\ref{c1x}).  Thus, there is $\alpha>0$ and a
constant $C>0$ such
that
$$\|u(t)-v(t)\|_{L^2(\Omega)}\le
	Ce^{-\alpha t}
\qquad\hbox{for all}\qquad t\ge 0.$$
Now use $v(t_1)$ as the initial condition
from which to make a future prediction.

Let $w$ be a solution to (\ref{nseqfunc})
with initial condition $w(t_1)=v(t_1)$.
Known results on continuous dependence
on initial conditions,
see, for example, \cite{constantin1988}, \cite{olson2011},
\cite{robinson2001} or \cite{temam1983}, imply
there is $\beta>0$ such that
$$
	\|w(t)-u(t)\|_{L^2(\Omega)}
	\le \|w(t_1)-u(t_1)\|_{L^2(\Omega)}
		e^{\beta (t-t_1)}
\qquad\hbox{for}\qquad t\ge t_1.$$
Therefore
$$
	\|w(t)-u(t)\|\le Ce^{-\alpha t_1+\beta T}<\epsilon
\qquad\hbox{for}\qquad t\in [t_1,t_1+T]$$
provided
	$\alpha t_1\ge \beta T +\ln(C/\epsilon)$.
Thus $w(t)$ predicts $u(t)$
with accuracy $\epsilon$ on the interval
$[t_1,t_1+T]$.

Work is currently
underway to numerically test
Theorem \ref{mainresp} in the
case of determining finite volume elements and nodes.
Of particular focus is how to tune the
parameter $\mu$.
If $\mu$ is very large the effects of ``spill over" into
the fine scales may become significant,
whereas if $\mu$ is small convergence
of the approximate solution may be slow or not happen at all.
Numerical simulations performed by
Gesho \cite{masa2013} confirm that
the   continuous data assimilation algorithm given by
equation (\ref{approxeqn})  directly works, without additional filtering,
for observational measurements at a discrete
set of nodal points, where $I_h$
is given by (\ref{interpnode}).
As with previous computational work (cf. \cite{olson2011},\cite{olson2003} and \cite{olson2008}) 
the approximating solution $v(t)$
converges to the reference solution $u(t)$
under much less stringent conditions than
required by our theory.

The main advantage of introducing a control term that
forces the approximate solution toward the reference
solution is that we can rely on the viscous dissipation, already
present in the dynamics, to filter the observational data (that is, to suppress the spatial oscillations, i.e. the ``spill over" into the fine scales, that are generated by the coarse-mesh stabilizing  term $\mu I_h(v)$).
In addition to working for a general class of interpolant observables
this technique also allows processing of observational data
which contains stochastic noise.
In particular, the same algorithm can be used to
obtain an approximation $v(t)$  that converges (in some sense) to the reference solution
$u(t)$,  to within an error of the order of    $\mu$ times the variance of the noise in the measurements.
This work \cite{bessaih2014} is in progress.

\section{Acknowledgements}

The work of A.A.~is supported in part by the DFG grants SFB-910 and SFB-947.
E.S.T.~is thankful to the kind hospitality of the Freie
Universit\"at Berlin, where this work was initiated. E.S.T.~also acknowledges the partial support of the
Alexander von Humboldt Stiftung/Foundation, the
Minerva Stiftung/Foundation, and the National Science Foundation grants DMS--1009950,
DMS--1109640 and DMS--1109645.

\vfill\eject

\end{document}